\documentclass[11pt, a4paper]{article}
\usepackage{courier}
\usepackage{blindtext} 
\usepackage{amsmath}
\usepackage{amssymb}
\usepackage[colorlinks=true,breaklinks]{hyperref}
\usepackage[hyphenbreaks]{breakurl}
\usepackage{xcolor}
\definecolor{c1}{rgb}{0,0,1}
\definecolor{c2}{rgb}{0,0.3,0.9}
\definecolor{c3}{rgb}{0.3,0.9}
\hypersetup{linkcolor={c1},citecolor={c2},urlcolor={c3}}
\usepackage{enumerate}
\usepackage{todonotes}
\usepackage{makeidx}
\makeindex
\usepackage[top=3.2cm,
bottom=3.5cm,
left=2.8cm,right=2.8cm]{geometry}
\usepackage{fancyhdr}

\def\XXint#1#2#3{{\setbox0=\hbox{$#1{#2#3}{\int}$ }
\vcenter{\hbox{$#2#3$ }}\kern-.6\wd0}}

\usepackage{amsthm}
\theoremstyle{plain}
\newtheorem{theorem}{Theorem}[section]

\theoremstyle{definition}

\theoremstyle{lemma}
\newtheorem{lemma}[theorem]{Lemma}

\theoremstyle{Remark}

\theoremstyle{proposition}

\theoremstyle{corollary}

\theoremstyle{example}

\theoremstyle{assumption}

\usepackage{systeme}
\usepackage{colonequals}
\usepackage{comment}
\usepackage{cite}
\begin{document}
\pagestyle{empty}

\title{Convergence to global equilibrium for the semiconductor Boltzmann equation}

\author{Gayrat Toshpulatov\thanks{Institut f\"ur Analysis und Numerik, Fachbereich Mathematik
und Informatik der Universit\"at M\"unster, Orl\'eans-Ring 10, 48149 M\"unster, Germany, {\tt gayrat.toshpulatov@uni-muenster.de}}}
\maketitle

\pagestyle{plain}
\begin{abstract}
The Boltzmann equation describing the transport of electrons in  semiconductor devices with an external electrostatic potential is considered  when the spatial variable is in a torus and the wave vector is in the Brillouin zone. We prove the exponential time decay of solutions towards the global equilibrium  in a weighted $L^2$ space. Our result holds for wide classes of energy functions of electrons and external electrostatic potentials, and
the estimates on the rate of convergence are explicit and constructive. We remove  the close-to-equilibrium assumption on the initial datum and the parabolic band approximation assumption that were required in previous works. The technique is based on the construction of a suitable highly non-linear Lyapunov
functional by modifying the relative entropy.  The analysis benefits from uniform bounds for the solution in terms of the global equilibrium. 
\end{abstract}
\textcolor{black}{\begin{small}\textbf{Keywords:} Semi-classical kinetic equations, Boltzmann equation for semiconductors,  convergence to equilibrium, Fermi-Dirac distribution, hypocoercivity, Lyapunov functional.
\end{small}}
\tableofcontents
\section{Introduction}
In semiconductor devices, electrical currents originate from the transport of electrons and holes. In semi-classical kinetic description, the statistical evolution of electrons  can be described by 
the following spatially inhomogeneous Boltzmann equation \cite{Book1, Book2, Book3}
\begin{equation*}
\begin{cases}
\displaystyle \partial_t f(t,x,k)+\frac{1}{\hbar}\nabla_{k}{\mathcal{E}}(k)\cdot \nabla_x f(t,x,k)-\frac{q}{\hbar}\nabla_x V(x)\cdot \nabla_k f(t,x,k)=Q(f)(t,x,k)\\
\displaystyle f(0,x,k)=f_{0}(x,k).
\end{cases}
\end{equation*}
Here the variables $t\geq 0$ and  $x \in \mathbb{T}^d$   stand for time and  position, respectively. In practice, $d$ is actually equal to two or three, however, in this paper,  $d$ can be any natural number.   The variable $k$ stands for wave vector and it belongs to the \emph{Brillouin zone} $B\subset \mathbb{R}^d$ associated with the underlying crystal lattice. The Brillouin zone $B$ is the elementary cell of the dual lattice and it will be identified with the  $d$-dimensional torus corresponding to the periodicity of the lattice. We have the property of $B$ that $k\in B$ implies $-k\in B.$ The constants $\hbar$ and $q$ are  respectively the Planck constant and  the positive elementary charge. The function $\mathcal{E}(k)$   describes the energy of electrons. The expression of $\mathcal{E}(k)$ is determined by acceptable solutions of the Schr\"odinger equation for the actual crystal potential, which varies markedly from solid to another. The function $V(x)$   describes a given external electrostatic potential.  In general, the electrostatic potential is  deduced from the distribution function through the Coulomb interaction, but for sake of simplicity, we shall assume here that it is given and does not depend on time $t.$  The unknown $f(t,x,k)\geq 0$ describes the phase space distribution function of particles. The operator $Q$  describes the interactions with the particles and the device, and it can be written as 
\begin{align*}
Q(f)&=\int_B\big( s(x,k',k)f'(1-f)-s(x,k,k')f(1-f')\big)dk',
\end{align*}
where $f=f(t,x,k)$ and $f'=f(t,x,k').$ The function $s(x,k',k)$ is non-negative and presents the probability for a particle to change its wave vector $k$ into another $k'$ during interaction at the position $x$. We assume that the \emph{principle of detailed balance} holds, according to which 
$$s(x,k,k')e^{\frac{\mathcal{E}(k')}{K_BT}}  =s(x,k',k) e^{\frac{\mathcal{E}(k)}{K_BT}},$$
 where $K_B$ is the Boltzmann constant and $T$ is the temperature of crystal. This relation is a sufficient condition for the  Fermi-Dirac distributions  to be in the null space of $Q$.    We denote $$\sigma(x,k,k')\colonequals s(x,k,k')e^{\frac{\mathcal{E}(k')}{K_BT}}.$$ Then the principle of detailed balance condition provides that $\sigma(x,k,k')=\sigma(x,k',k)$. 

 The wave
vector $k$ is usually taken in a bounded set, the Brillouin zone $B$, which is  the most physically
relevant case.  However,  when the particles under consideration have energies close to an extremum of an energy band, then the \emph{parabolic band approximation}  \cite{Book1, Book2} is used, which consists in taking
\begin{equation*}
\mathcal{E}(k)=\frac{\hbar}{2m}|k|^2\, \,\text{ and } \, \, B=\mathbb{R}^d,
\end{equation*}
where $m$ is the effective mass of the particle.
 To include the parabolic band approximation case, in this paper,  $B$ is a bounded set (which is identified with the torus corresponding to the periodicity of the lattice)  or $\mathbb{R}^d.$ If $B=\mathbb{R}^d$, then $\mathcal{E}(k)$ is a quadratic polynomial as above.  

For simplicity, we set all physical constants to unity: $\hbar=q=K_B=T=m=1.$ This is not a restriction since one can check that our results hold without this condition. Therefore, we shall consider the normalized equation
\begin{equation}\label{Eq}
\begin{cases}
\displaystyle \partial_t f(t,x,k)+\nabla_{k}{\mathcal{E}}(k)\cdot \nabla_x f(t,x,k)-\nabla_x V(x)\cdot \nabla_k f(t,x,k)=Q(f)(t,x,k)\\
\displaystyle f(0,x,k)=f_{0}(x,k).
\end{cases}
\end{equation}
with \begin{equation*}
Q(f)=\int_{B}\sigma(x,k,k')\big(e^{-\mathcal{E}(k)}(1-f)f'-e^{-\mathcal{E}(k')}(1-f')f\big) dk'.
\end{equation*} 
 
Equation \eqref{Eq} has several properties following standard physical considerations. 
  The factors $(1-f)$ and $(1-f')$ appearing in $Q$ take into account the Pauli exclusion principle. The distribution function is supposed to satisfy  
\begin{equation*}
0\leq f\leq 1.
\end{equation*}

Whenever $f(t,x,k)$ is a (well-behaved) solution of \eqref{Eq}, one has  \textit{global conservation of mass} 
\begin{equation*}
\int_{\mathbb{T}^3}\int_Bf(t,x,k)dkdx=\int_{\mathbb{T}^3}\int_B f_0(x,k)dkdx, \, \, \, \, \forall\, t\geq 0.
\end{equation*}

Equation \eqref{Eq} is \emph{dissipative} in the sense that the following entropy functional  decreases under the time-evolution of the solution $f$:   For functions $f$ satisfying $0< f<1$ ($f$ is not necessarily the solution), we define  the \emph{entropy functional}  
\begin{equation}\label{Hch}\mathrm{H}(f)\colonequals \int_{\mathbb{T}^d}\int_{B}\big(f\log{f}+(1-f)\log{(1-f)+f(\mathcal{E}(k)+V(x))}\big)dkdx.
\end{equation} 
If we have the existence of a solution $f\in (0,1)$ which is regular enough, one can check (see \eqref{dtH} below), for all $t>0,$
\begin{align}\label{dt H2}
\frac{d}{dt}\mathrm{H}(f(t))=-\frac{1}{2}\int_{\mathbb{T}^d}\int_{B}\int_{B}\sigma (x,k,k') e^{-\mathcal{E}-\mathcal{E}'-V}(1-f)(1-f')(F-F')(\log{F}-\log{F'})dk'dkdx,
\end{align}
where  $\mathcal{E}=\mathcal{E}(k),$ $\mathcal{E}'=\mathcal{E}(k')$, $ F\colonequals \frac{f(t,x,k)}{e^{-\mathcal{E}(k)-V(x)}(1-f(t,x,k))},$ and $ F'\colonequals \frac{f(t,x,k')}{e^{-\mathcal{E}(k')-V(x)}(1-f(t,x,k'))}.$  
Since the function $ s\mapsto  \log{s}$ is increasing  for $s>0$, we have $$(F-F')(\log{F}-\log{F'})\geq 0,$$ hence
$$\frac{d}{dt}\mathrm{H}(f(t))\leq 0, \, \, \, \forall\, t>0.$$
 This means that   $\mathrm{H}(f(t))$ decreases under the time-evolution of the solution $f$. Moreover, the entropy dissipation functional  (the integral on the right hand side of \eqref{dt H2})  vanishes only if $F=F'$. This means that $F$ does not depend on $k$ and so 
\begin{equation}\label{loc}
f(t,x,k)=\frac{1}{1+e^{\mathcal{E}(k)+V(x)-\mu(t,x)}}
\end{equation}
for some function $\mu(t,x).$ Hence, any function  $f$ in the form of \eqref{loc} is a \emph{local equilibrium} for \eqref{Eq}. For such $f $ both left and right hand sides of \eqref{Eq} vanish if $\mu$ is constant. Hence, the unique  \emph{global equilibrium} for \eqref{Eq} is 
\begin{equation}\label{eqilib}
f_{\infty}(x,k)=\frac{1}{1+e^{\mathcal{E}(k)+V(x)-\mu_{\infty}}},
\end{equation}
where $\mu_{\infty}>0$ is determined by mass conservation $$\displaystyle \int_{\mathbb{T}^d}\int_Bf_{0}(x,k)dkdx=\int_{\mathbb{T}^d}\int_{ B}f_{\infty}(x,k)dkdx.$$  The function $f_{\infty}(x,k)$ is called the  Fermi-Dirac distribution.

On the basis of the decay of $\mathrm{H}(f(t)),$ one can conjecture that  $f(t)$ converges to the global equilibrium  $f_{\infty}$ as $t\to \infty.$   Thus, it is a fundamental problem to prove (or disprove) this convergence and to estimate the convergence rate. Our goal is to  prove this convergence and to derive explicit and constructive estimates for convergence rate. 

There are many works concerning the semiconductor Boltzmann equation \eqref{Eq}.  We refer the books \cite{Book1, Book2, Book3} for the physical background of the equation  and its applications. The existence of weak solutions to the semiconductor Boltzmann equation was proved by Poupaud \cite{Pop.}. 
 In the case of the parabolic band approximation (i.e., $\mathcal{E}(k)=|k|^2$ and $B=\mathbb{R}^d$), the  existence  of smooth global-in-time solutions  was proven in \cite{Mus, Exist2, P.S, Exist1, Exist3}. In all these works, the interaction rate  $s(x,k,k')$ was assumed to be smooth. Majorana et al.\,\cite{Maj1,  Maj2, Maj3} proved the existence of solutions of the spatially homogeneous  equation allowing for non-smooth interaction rates.

 The diffusion limit of the semiconductor Boltzmann equation has been studied by several authors. The asymptotics leads to the so-called energy transport models. We refer \cite{Asym3, Des, Asym5,   G.P,  Asym2, Jun,  Asym4, Asym6, Asym1, P.S} for the derivation and the  mathematical analysis of these models.

Proving decay of solutions to equilibrium  is one of the fundamental problems for kinetic equations.  For spatially homogeneous kinetic equations it is well-known that explicit estimates on convergence to equilibrium can be obtained by the direct study of the relative entropy and its  dissipation functional. With the help of so-called Sobolev inequalities one can obtain a Gr\"onwall inequality for the relative entropy, which implies convergence of the solution to the global equilibrium. Unfortunately,  this idea can not be used for spatially inhomogeneous kinetic equations. The reason is that the collision operator acts only on the velocity variable, which causes to loose information about $x-$direction in the dissipation functional. To deal with this problem, in  recent years, many new methods, so called hypocoercivity methods, are introduced to study the long-time behavior of spatially inhomogeneous kinetic equations.
The challenge of hypocoercivity is to understand the interplay between the collision operator that provides dissipativity in the velocity variable and the transport one which is conservative, in order to obtain global dissipativity for the whole problem. There are several hypocoercivity methods: see \cite{D.V, V, Neu.Schmeis} for the entropy-entropy dissipation method,  \cite{NeuMou, Herau1, V, Achleitner2015,  AT1, AT2} for the $H^1$-hypocoercivity method, \cite{Herau, DMS, Max, BHR, Car, APT, Tosh} for the $L^2$-hypocoercivity method, and  \cite{Guo, Ebe, Can,  Cao} for some other methods.
 
Concerning the long time behavior of the semiconductor Boltzmann equation \eqref{Eq}, there are  results only in the case of the parabolic band approximation (i.e., $\mathcal{E}(k)=|k|^2$ and $B=\mathbb{R}^d$) and without the electrostatic potential (i.e., $V(x)$ is identically zero): Neumann and Schmeiser \cite{Neu.Schmeis} obtained   a polynomial rate of convergence to equilibrium. Their proof is based on the entropy-entropy dissipation method  and uniform regularity bounds on the solution. Using the $H^1$-hypocoercivity method and assuming that the initial data is close to equilibrium (i.e., $f_0-f_{\infty}$ is sufficiently small) in a Sobolev space, an exponential decay result was obtained in \cite{NeuMou}. Recently, Chen,  Liu, and Wan \cite{Chen} considered the semiconductor Boltzmann equation with the diffusive reflection boundary condition. They  proved the exponential convergence of the solution to the global equilibrium in $L^{\infty}$ when the initial   data is close to equilibrium.
 Pirner and the author \cite{P.T} proved that when initial data is bounded in terms of the equilibrium, then the solution converges to the equilibrium exponentially.

As mentioned above, the long time behavior of solutions has been studied in the case of the parabolic band approximation  and without the electrostatic potential. In this paper, we shall improve these previous results. We consider the most physically relevant case by letting the wave
vector $k$ be in a bounded set, the Brillouin zone $B$.  The energy function $\mathcal{E}(k)$ can be non-quadratic, and  we only require that  the matrix $\int_B \frac{\partial^2\mathcal{E}(k)}{\partial k^2}e^{-\mathcal{E}(k)} dk$ is positive definite. We assume the electrostatic potential $V$ is a given smooth function.  We do not require the close-to-equilibrium assumption on the initial data.  We prove exponential decay of the solution  to the global equilibrium.  Our results also include  the parabolic band approximation case.
 The proof is based on the construction of  a suitable  Lyapunov functional for \eqref{Eq}.  This functional has the following form
\begin{align}\label{E(f)}
\mathrm{E}(f(t))\colonequals \int_{\mathbb{T}^d}\int_B& \left( f(t,x,k)\log\frac{f(t,x,k)}{f_{\infty}(x,k)}+(1-f(t,x,k))\log\frac{1-f(t,x,k)}{1-f_{\infty}(x,k)}\right)dkdx\nonumber \\
&+\delta \int_{\mathbb{T}^d}\int_B \nabla_x \phi(t,x) \cdot \nabla_k \mathcal{E}(k)\log\left(\frac{f(t,x,k)}{1-f(t,x,k)}\right) e^{-\mathcal{E}(k)}dk dx,
\end{align}
 where $\delta>0$  is a small constant and  the function $\phi $ is the solution of the elliptic equation 
 \begin{equation*}
-\mathrm{div}_x(A\nabla_x \phi(t,x))=\int_B \big(f(t,x,k)-f_{\infty}(x,k)\big)dk,\, \, \, \int_{\mathbb{T}^d} \phi(t,x) dx=0, \, \, \, x\in \mathbb{T}^d, \, \, t\geq 0
\end{equation*}
with $$A\colonequals \int_B \frac{\partial^2\mathcal{E}(k)}{\partial k^2}e^{-\mathcal{E}(k)} dk\in \mathbb{R}^{d\times d}.$$ One can check that the first term on the right hand side of \eqref{E(f)}  is equal to the \emph{relative entropy} $$\mathrm{H}(f|f_{\infty})\colonequals \mathrm{H}(f)-\mathrm{H}(f_{\infty}).$$ Hence, one can consider $\mathrm{E}(f)$ as a modification of the relative entropy. The second term in \eqref{E(f)} is chosen so that,   if the initial data is bounded in terms of the equilibrium (see \eqref{b1} below), then this functional is equivalent to the square of the weighted $L^2$-norm
$$\int_{\mathbb{T}^d}\int_B\big(f(t,x,k)-f_{\infty}(x,k)\big)^2e^{\mathcal{E}(k)}dxdk$$
and it satisfies 
$$\frac{d}{dt } \mathrm{E}(f(t))\leq -2\lambda  \mathrm{E}(f(t)),\,\,\, \forall\,t>0$$
for some $\lambda>0$ (see the proof of Theorem \ref{main th}).  This inequality lets us conclude that the solution $f(t,x,v)$ converges to the equilibrium $f_{\infty}(x,k)$ exponentially in the above $L^2$-norm.  

Compared to previous hypercoercivity methods,  the Lyapunov functional $\mathrm{E}(f)$ seems new and can be useful to study long time behavior of other semi-classical kinetic equations.  
The results of this work provide an extension to the results of \cite{ NeuMou, Neu.Schmeis, Chen, P.T} because  we do not require  the parabolic band approximation and close-to-equilibrium assumptions,  and the electrostatic potential does not have to be identically zero.


The organization of this paper is as follows. In Section 2, we state the main results. In Section 3, we establish existence of a unique distributional solution. In Sections 4 and 5, we study the relative entropy and  its dissipation, and prove some estimates involving a non-linear  projection operator. In section 6, we  construct a Lyapunov functional and present the proof of our main result.

\section{Main results}
 Our main result is the following:
\begin{theorem}\label{main th}
Let  $\mathcal{E}:B\to \mathbb{R},$  $V:\mathbb{T}^d\to \mathbb{R}$, and $\sigma :\mathbb{T}^d\times B\times B\to \mathbb{R}$ be smooth functions.
Assume that there exist constants    $\sigma_{+}>0,$ $\mu_{1}>0,$ and  $ \mu_2>0$ such that 
\begin{equation*}
   0 \leq \sigma(x,k,k')\leq \sigma_{+}, \, \, \forall\, x\in \mathbb{T}^d,\,\,\, \forall\, k,k'\in B,
\end{equation*}
and 
\begin{equation*}
\frac{1}{1+e^{\mathcal{E}(k)+V(x)-\mu_1}}\leq f_{0}(x,k)\leq \frac{1}{1+e^{\mathcal{E}(k)+V(x)-\mu_2}}, \, \, \forall\, x \in \mathbb{T}^d, \,  \forall\, k\in B.
\end{equation*}
Then there is a unique distributional solution $f\in C([0,\infty), L^1(\mathbb{T}^d\times B))$ to \eqref{Eq} satisfying the same bounds
\begin{equation}\label{b1}
\frac{1}{1+e^{\mathcal{E}(k)+V(x)-\mu_1}}\leq f(t,x,k)\leq \frac{1}{1+e^{\mathcal{E}(k)+V(x)-\mu_2}}
\end{equation}
for all $t\geq 0,$ $x \in \mathbb{T}^d,$ and $k\in B.$  Moreover, assume that  the matrix
$$\int_B \frac{\partial^2\mathcal{E}(k)}{\partial k^2}e^{-\mathcal{E}(k)} dk\in \mathbb{R}^{d\times d}$$ is positive definite, and there is a constant $\sigma_{-}>0$ such that
\begin{equation*}
   \sigma_{-} \leq \sigma(x,k,k'), \, \, \forall\, x\in \mathbb{T}^d,\,\,\, \forall\, k,k'\in B.
\end{equation*}
 Then there exist  explicitly computable constants $\lambda>0$ and $ c>0$  such that  the solution $f$  satisfies
\begin{equation}\label{lamda}
\int_{\mathbb{T}^d}\int_B\big(f(t,x,k)-f_{\infty}(x,k)\big)^2e^{\mathcal{E}(k)}dkdx\leq c e^{-2\lambda t}\int_{\mathbb{T}^d}\int_B\big(f_0(x,k)-f_{\infty}(x,k)\big)^2e^{\mathcal{E}(k)}dkdx
\end{equation}
for all $t\geq 0.$ 
\end{theorem}

Note that \eqref{lamda} shows that $f(t,x,k)$ converges exponentially to $f_{\infty}(x,k)$ in the weighted $L^2$-norm $$\sqrt{\int_{\mathbb{T}^d}\int_B\big(f(t,x,k)-f_{\infty}(x,k)\big)^2e^{\mathcal{E}(k)}dxdk}. $$ If the Brillouin zone $B$ is bounded, then $e^{\mathcal{E}}$ is bounded from below and above, hence this weighted $L^2$-norm is equivalent to the usual $L^2$-norm.

We will show in the proof of Lemma \ref{lem:eq} that the relative entropy $\mathrm{H}(f(t)|f_{\infty})$ is equivalent to the square of  the above weighted $L^2$-norm. Therefore, \eqref{lamda} is equivalent to
\begin{equation*}
\mathrm{H}(f(t)|f_{\infty})\leq \tilde{c}e^{-2\lambda t}\mathrm{H}(f_0|f_{\infty}),\,\,\,\, \forall\, t\geq 0,
\end{equation*}
for some constant $\tilde{c}>0$.

As mentioned above, in Theorem \ref{main th}  $B$ is a bounded set (which is identified with the torus corresponding to the periodicity of the lattice)  or $\mathbb{R}^d.$
 The case $B=\mathbb{R}^d$ and $\mathcal{E}(k)=|k|^2$ corresponds to the parabolic band approximation.  In  this case we can compute 
$$\int_B \frac{\partial^2\mathcal{E}(k)}{\partial k^2}e^{-\mathcal{E}(k)} dk=\left(\frac{1}{d}\int_{\mathbb{R}^d} |k|^2e^{-|k|^2} dk\right) \mathrm{Id}=\frac{\pi^{d/2}}{2}\mathrm{Id},$$ where $\mathrm{Id}$ denotes the identity matrix in $\mathbb{R}^{d\times d}$.  Since this matrix is positive definite, Theorem \ref{main th} also holds for the parabolic band approximation case. One can check that the examples of $\mathcal{E}(k)$ in \cite[Appendix A]{Maj1} also satisfy this condition.

There are more general models than   Equation \eqref{Eq}: As done in previous works  \cite{P.S, G.P,Neu.Schmeis, NeuMou, Asym6, Asym2, Asym4, Asym3, Asym5, Asym1,Chen, P.T},  we assume in Theorem \ref{main th} that the interaction rate $\sigma(x,k,k')$ is smooth and  bounded from below and above. In some cases the interaction rate  can be highly non-smooth, see  \cite[Chapter 4]{Book2}, \cite{Maj1, Maj2, Maj3}.
 Also, consideration of
particle-particle interactions gives rise to a nonlinear collision operator of fourth order, see \cite{Des}, \cite[Chapter 4]{Book2}. If we take into account the generation and recombination of electrons and holes, we can  obtains a
coupled system for the distributions of electrons and holes, see \cite{Pop., G.P}. Extending our results to such models  would of
course be interesting. 
Yet, this extension is not within reach so far and will be a
matter of further study. We expect that the techniques used in this paper can
be useful to study the long-time behavior of these models.
\section{Existence and uniqueness}
In this section we prove the existence part of Theorem \ref{main th}.
\begin{theorem}\label{existence}
Under the assumption of Theorem \ref{main th},  there is a unique distributional solution $f\in C([0,\infty), L^1(\mathbb{T}^d\times B))$ to \eqref{Eq} satisfying the bounds \eqref{b1}.
\end{theorem}
\begin{proof} We first define the following operators
$$\mathcal A(f)\colonequals e^{-\mathcal{E}(k)}\int_B \sigma(x,k,k') f'dk', $$
$$\mathcal B(f)\colonequals \int_B \sigma(x,k,k') (1-f')e^{-\mathcal{E}'}dk'.$$
With these notations $f$ is a solution if it satisfies  
$$\partial_t f+\nabla_{k}{\mathcal{E}}\cdot \nabla_x f-\nabla_x V\cdot \nabla_k f+[\mathcal A(f)+\mathcal B(f)]f=\mathcal A(f).$$
We define the set $$\mathcal{V}\colonequals \big\{f\in C([0,\infty), L^1(\mathbb{T}^d\times B)):\,\,f\,\,\text{satisfies}\,\,\eqref{b1} \big\}.$$ For a given $f\in \mathcal{V}$, let $g$ be the solution of the linear transport equation
$$\begin{cases}\partial_t g+\nabla_{k}{\mathcal{E}}\cdot \nabla_x g-\nabla_x V\cdot \nabla_k g+[\mathcal A(f)+\mathcal B(f)]g=\mathcal A(f)\\
g(0,x,k)=f_{0}(x,k).
\end{cases}
$$
We then define the map $\Gamma(f)=g.$ Obviously, a fixed point of $\Gamma$ correspond to a solution of the equation.
 We  show that $\Gamma$ maps $\mathcal{V}$ to itself.
Let  $r\colonequals g-\frac{1}{1+e^{\mathcal{E}+V-\mu_1}}.$
 Then we have 
\begin{align*}
\partial_t r+\nabla_{k}{\mathcal{E}}\cdot \nabla_x r-\nabla_x V\cdot \nabla_k r+[\mathcal A(f)+\mathcal B(f)]r
= \left(1-\frac{1}{1+e^{\mathcal{E}+V-\mu_1}}\right) \mathcal A(f)-\frac{\mathcal B(f)}{1+e^{\mathcal{E}+V-\mu_1}}.
\end{align*}
Since $\mathcal A(f)$ and $\mathcal B(f)$ are respectively  increasing and decreasing functions of $f$, we have  \begin{align*}
\left(1-\frac{1}{1+e^{\mathcal{E}+V-\mu_1}}\right)&\mathcal A(f)-\frac{\mathcal B(f)}{1+e^{\mathcal{E}+V-\mu_1}}\\ &\geq \left(1-\frac{1}{1+e^{\mathcal{E}+V-\mu_1}}\right)\mathcal A\left(\frac{1}{1+e^{\mathcal{E}+V-\mu_1}}\right)- \frac{\mathcal B\left(\frac{1}{1+e^{\mathcal{E}+V-\mu_1}}\right)}{1+e^{\mathcal{E}+V-\mu_1}}\\
&=Q\left(\frac{1}{1+e^{\mathcal{E}+V-\mu_1}}\right)=0. 
\end{align*}
This estimate shows 
$$\partial_t r+\nabla_{k}{\mathcal{E}}\cdot \nabla_x r-\nabla_x V\cdot \nabla_k r+[\mathcal A(f)+\mathcal B(f)]r\geq 0.$$
Since  $r(0,x,k)\geq 0$, we conclude $r(t,x,k)\geq 0.$  This means  $g\geq \frac{1}{1+e^{\mathcal{E}+V-\mu_1}}$.  Similar arguments as above imply  $g\leq  \frac{1}{1+e^{\mathcal{E}+V-\mu_2}}$ and so $g\in \mathcal{V}.$

 Next, we show that $\Gamma$ is a contraction for sufficiently small time
intervals.
Let $f$ and  $h$ be in $ \mathcal{V}.$ We define $g\colonequals \Gamma(f)$ and $u\colonequals \Gamma(h).$ 
Then we have
\begin{align*}
\partial_t (g-u)+\nabla_{k}{\mathcal{E}}\cdot \nabla_x (g-u)&-\nabla_x V\cdot \nabla_k (g-u)+[\mathcal A(f)+\mathcal B(f)](g-u)\\
= &(1-u)\mathcal A(f-h)-u\mathcal B(f-h).
\end{align*}
We multiply this equation by $\mathrm{sign}( g-u)$
\begin{align*}
\partial_t |g-u|+\nabla_{k}{\mathcal{E}}\cdot \nabla_x |g-u|&-\nabla_x V\cdot \nabla_k |g-u|+[\mathcal A(f)+\mathcal B(f)]|g-u|\\
= & \mathrm{sign}( g-u) [(1-u)\mathcal A(f-h)-u \mathcal B(f-h)].
\end{align*}
 We  integrate  and use $\mathcal A(f)+\mathcal B(f)\geq 0$
 \begin{align*}
 \frac{d}{dt}||g-u||_{L^1(\mathbb{T}^d\times B)}\leq ||  (1-u)A(f-h)-uB(f-h)||_{L^1(\mathbb{T}^d\times B)}.
 \end{align*}
Using $0<u<1$ we  estimate
 $$||(1-u)\mathcal A(f-h)||_{L^1(\mathbb{T}^d\times B)}\leq ||\mathcal A(f-h)||_{L^1(\mathbb{T}^d\times B)}\leq \sigma_+\left(\int_B e^{-\mathcal{E}(k)}dk\right) ||f-h||_{L^1(\mathbb{T}^d\times B)}. $$
 Since $u\in \mathcal{V}$ and $u\leq \frac{1}{1+e^{\min_{\mathbb{T}^d}\{V\}+\mathcal{E}-\mu_2}},$  we have 
 $$||u\mathcal B(f-h)||_{L^1(\mathbb{T}^d\times B)}\leq \sigma_+ e^{-\min_B\{\mathcal{E}\}}\left(\int_{B}\frac{dk}{1+e^{\min_{\mathbb{T}^d}\{V\}+\mathcal{E}(k)-\mu_2}}\right) ||f-h||_{L^1(\mathbb{T}^d\times B)}. $$
 This leads 
 \begin{align*}
&|| \Gamma (f)(t)-\Gamma (h)(t)||_{L^1(\mathbb{T}^d\times B)}\\
 &\leq \left[\int_B \sigma_+ e^{-\mathcal{E}(k)}dk+\int_{B}\frac{\sigma_+ e^{-\min_B\{\mathcal{E}\}}dk}{1+e^{\min_{\mathbb{T}^d}\{V\}+\mathcal{E}(k)-\mu_2}}\right]\int_0^t ||f(s)-h(s)||_{L^1(\mathbb{T}^d\times B)}ds.
 \end{align*}
 Then the fixed point arguments provide the existence of a unique solution.
\end{proof}

\section{Relative Entropy}
In this section we show that  the relative entropy decays, and we get bounds on the collision operator with respect to the dissipation functional.

For non-negative function $f$ we define the relative entropy 
$$\mathrm{H}(f|f_{\infty})\colonequals \int_{\mathbb{T}^d}\int_B\left( f\log\frac{f}{f_{\infty}}+(1-f)\log\frac{1-f}{1-f_{\infty}}\right)dkdx.$$
We note, if $f$ and $f_{\infty}$ have the same mass, then
$$\mathrm{H}(f|f_{\infty})=\mathrm{H}(f)-\mathrm{H}(f_{\infty}).$$
\begin{lemma} Let  $\mathcal{E}:B\to \mathbb{R},$  $V:\mathbb{T}^d\to \mathbb{R}$, and $\sigma :\mathbb{T}^d\times B\times B\to \mathbb{R}$ be smooth functions. Let $f$ be in $W^{1,\infty}([0,T]\times \mathbb{T}^d\times B)$ for all $T>0$ and $0<f<1.$ If $f$ is a solution of \eqref{Eq}, then 
\begin{align} \label{dtH}
\frac{d}{dt}\mathrm{H}(f(t)|f_{\infty})=
-\frac{1}{2}\int_{\mathbb{T}^d}\int_{ B}\int_{ B}&\sigma(x,k,k')e^{-\mathcal{E}-\mathcal{E}'-V}(1-f)(1-f')\times \nonumber \\
&\times (F-F')(\log F-\log F') dk'dkdx\leq 0,
\end{align}
where $F\colonequals \frac{f}{(1-f)e^{-\mathcal{E}-V}}.$
\end{lemma}
\begin{proof}
We compute 
\begin{align*}
\frac{d}{dt}\mathrm{H}(f(t)|f_{\infty})&=\int_{\mathbb{T}^d}\int_B\left(\partial_tf \log\frac{f}{1-f}-\partial_tf \log\frac{f_{\infty}}{1-f_{\infty}}\right)dxdk\\&=\int_{\mathbb{T}^d}\int_B\left(\partial_tf \log\frac{f}{(1-f)}+\partial_tf (\mathcal{E}+V-  \mu_{\infty})\right)dxdk.
\end{align*}
The conservation of mass implies  $$\int_{\mathbb{T}^d}\partial_tf  \mu_{\infty}dxdk=0.$$
 Integrating by parts one can easily show 
$$\int_{\mathbb{T}^d}\int_B\left(\nabla_k \mathcal{E}(k)\cdot \nabla_x f-\nabla_x V(x)\cdot \nabla_k f\right) (\mathcal{E}(k)+V(x)))dkdx=0.  $$
We can write $$\nabla_x f \log\frac{f}{1-f}=\nabla_x\left(f\log{f}+(1-f)\log{(1-f)}\right)$$ and $$\nabla_k f \log\frac{f}{(1-f)}=\nabla_k\left(f\log{f}+(1-f)\log{(1-f)}\right).$$
These equations and integration by parts show
$$\int_{\mathbb{T}^d}\int_B\left(\nabla_k \mathcal{E}(k)\cdot \nabla_x f-\nabla_x V(x)\cdot \nabla_k f\right) \log\frac{f}{1-f}dkdx=0.  $$ 
Hence, we get 
\begin{align*}
\frac{d}{dt}\mathrm{H}(f(t)|f_{\infty})=\int_{\mathbb{T}^d}\int_BQ(f)\log\frac{f}{(1-f)e^{-\mathcal{E}(k)-V(x)}}dxdk.
\end{align*}
By recalling $F\colonequals \frac{f}{(1-f)e^{-\mathcal{E}(k)-V(x)}}$ we compute 
\begin{align*}
&\int_{\mathbb{T}^d}\int_B Q(f)\log{F}dxdk\\&=\int_{\mathbb{T}^d}\int_{ B}\left(\int_{ B}\sigma(x,k,k')e^{-\mathcal{E}-\mathcal{E}'-V}(1-f)(1-f')(F'-F)\log F dk'\right)dkdx. 
\end{align*}
In this equation we change the order of the integrals with respect to $k'$ and $k$, and use  $\sigma(x,k,k')=\sigma(x,k',k)$
\begin{align*}
&\int_{\mathbb{T}^d}\int_BQ(f)\log{F}dxdk\\&=\int_{\mathbb{T}^d}\int_{ B}\left(\int_{ B}\sigma(x,k,k')e^{-\mathcal{E}-\mathcal{E}'-V}(1-f)(1-f')(F-F')\log F' dk'\right)dkdx. 
\end{align*}
We add the last two equations and obtain
 \begin{align*}
&\int_{\mathbb{T}^d\times B}Q(f)\log{F}dxdk\\&=-\frac{1}{2}\int_{\mathbb{T}^d}\int_{ B}\int_{ B}\sigma(x,k,k')e^{-\mathcal{E}-\mathcal{E}'-V}(1-f)(1-f')(F-F')(\log F-\log F') dk'dkdx\leq 0. 
\end{align*}
\end{proof}
We define the dissipation function
\begin{align*}
\mathcal{D}(f)\colonequals
\frac{1}{2}\int_{\mathbb{T}^d}\int_{ B}\int_{ B}\sigma(x,k,k')e^{-\mathcal{E}-\mathcal{E}'-V}(1-f)(1-f')(F-F')(\log F-\log F') dk'dkdx.
\end{align*}
With this notation  we can write \eqref{dtH} as
$$\frac{d}{dt}\mathrm{H}(f(t)|f_{\infty})=-\mathcal{D}(f(t))\leq 0.$$
\begin{lemma} Let  $\mathcal{E}:B\to \mathbb{R},$  $V:\mathbb{T}^d\to \mathbb{R}$, and $\sigma :\mathbb{T}^d\times B\times B\to \mathbb{R}$ be smooth functions. Assume that there exists a constant    $\sigma_{+}>0$  such that 
\begin{equation*}
   0 \leq \sigma(x,k,k')\leq \sigma_{+}, \, \, \forall\, x\in \mathbb{T}^d,\,\,\, \forall\, k,k'\in B.
\end{equation*} 
Then, for functions $f$ satisfying \eqref{b1}, there exists a constant $C_1>0$ such that
\begin{align}\label{Q<}
\int_{\mathbb{T}^d}\int_B|Q(f)|^2e^{\mathcal{E}}dkdx \leq C_1 \mathcal{D}(f).
\end{align}
\end{lemma}
\begin{proof}
The H\"older inequality provides 
\begin{align*}
&\int_{\mathbb{T}^d}\int_B|Q(f)|^2e^{\mathcal{E}+V}dkdx\\
&= \int_{\mathbb{T}^d}\int_B\left( \int_B \sigma(x,k,k') e^{-\mathcal{E}-\mathcal{E}'-V}(1-f)(1-f')(F'-
F)dk'\right)^2e^{\mathcal{E}+V}dkdx\\
&\leq \int_{\mathbb{T}^d}\int_B\left( \int_B \sigma(x,k,k')e^{-\mathcal{E}-\mathcal{E}'-V} (1-f)(1-f')dk'\right)\times \\
&\,\,\,\,\,\,\,\,\,\,\,\,\,\,\,\,\,\,\,\,\,\,\,\,\,\,\,\,\,\,\,\,\,\, \times \left( \int_B \sigma(x,k,k') e^{-\mathcal{E}-\mathcal{E}'-V}(1-f)(1-f')(F'-
F)^2dk'\right)e^{\mathcal{E}+V}dkdx.
\end{align*}
Since $\sigma(x,k',k)\leq \sigma_+,$ $1-f\leq 1,$ and $1-f'\leq 1,$ we have
\begin{align*}
&\int_{\mathbb{T}^d}\int_B|Q(f)|^2e^{\mathcal{E}+V}dkdx\\
&\leq \sigma_+ \left( \int_B  e^{-\mathcal{E}(k')}dk'\right)\int_{\mathbb{T}^d}\int_B \int_B \sigma(x,k,k') e^{-\mathcal{E}-\mathcal{E}'-V}(1-f)(1-f')(F'-
F)^2dk'dkdx\\
&=2\sigma_+ \left( \int_B  e^{-\mathcal{E}(k')}dk'\right)\mathcal{D}(f).
\end{align*}
This proves the claimed estimate with $C_1\colonequals 2\sigma_+e^{-\min_{\mathbb{T}^d}\{V\}}  \int_B  e^{-\mathcal{E}(k')}dk'$.
\end{proof}

\section{Projection operator }
We  denote  the macroscopic densities
$$\rho_{\infty}(x)\colonequals \int_{B}f_{\infty}(x, k)dk$$
and 
$$\rho(t,x)\colonequals \int_{B}f(t,x,k)dk.$$
 We define  a nonlinear projection operator
\begin{equation}\label{pr1}
\Pi f(t,x,k)\colonequals \frac{1}{1+e^{\mathcal{E}(k)+V(x)-\mu(x,t)}}, 
\end{equation}
where $\mu(t,x)$ is chosen by the condition
\begin{equation*}
\int_{B}\Pi f(t,x,k)dk=\rho(t,x).
\end{equation*}
In the following we prove some estimates involving the projection operator $\Pi.$
\begin{lemma} Let  $\mathcal{E}:B\to \mathbb{R}$ and  $V:\mathbb{T}^d\to \mathbb{R}$ be smooth functions.
Then, for functions $f$ satisfying \eqref{b1}, there are constant ${C}_2>0,$ ${C}_3>0,$ and ${C}_4>0$ such that 
\begin{equation}\label{rr}
{C}_2(\Pi f-f_{\infty})^2 e^{\mathcal{E}}\leq (\rho-\rho_{\infty})^2 e^{-\mathcal{E}}\leq {C}_3(\Pi f-f_{\infty})^2e^{\mathcal{E}}.
\end{equation}
and
\begin{equation}\label{rk}
{C}_4(\Pi f-f_{\infty})^2e^{\mathcal{E}}\leq  ({\mu}-{\mu_{\infty}})(\rho-\rho_{\infty})e^{-\mathcal{E}}.
\end{equation}
\end{lemma}
\begin{proof}
By the definition of $\Pi$ we have
  $$\rho-\rho_{\infty}=\int_{B}(\Pi f-f_{\infty})dk=
(e^{\mu}-e^{\mu_{\infty}})e^{-V}\int_{B}(1-\Pi f)(1-f_{\infty})e^{-\mathcal{E}}dk$$
and $$e^{\mu}-e^{\mu_{\infty}}=\frac{\Pi f-f_{\infty}}{(1-\Pi f)(1-f_{\infty})e^{-\mathcal{E}-V}}.$$ 
These equations imply  
\begin{equation}\label{r-r}(\rho-\rho_{\infty})^2=\frac{(\Pi f-f_{\infty})^2 e^{2\mathcal{E}} }{(1-\Pi f)^2(1-f_{\infty})^2}\left(\int_{B}(1-\Pi f)(1-f_{\infty})e^{-\mathcal{E}}dk\right)^2
\end{equation}
and 
\begin{align}\label{r-m}
(\rho-\rho_{\infty})(e^{\mu}-e^{\mu_{\infty}})&=\frac{(\Pi f-f_{\infty})^2 e^{2\mathcal{E}} e^{V}}{(1-\Pi f)^2(1-f_{\infty})^2}\int_{B}(1-\Pi f)(1-f_{\infty})e^{-\mathcal{E}}dk\nonumber \\
& \geq  \frac{(\Pi f-f_{\infty})^2 e^{2\mathcal{E}} e^{\min_{\mathbb{T}^d}\{V\}}}{(1-\Pi f)^2(1-f_{\infty})^2}\int_{B}(1-\Pi f)(1-f_{\infty})e^{-\mathcal{E}}dk.
\end{align}
Because of the bound \eqref{b1} we can estimate $$0< \frac{e^{\min_B \{\mathcal{E}\}+\min_{\mathbb{T}^d} \{V\}-\mu_2}}{1+e^{\min_B \{\mathcal{E}\}+\min_{\mathbb{T}^d} \{V\}-\mu_2}}\leq 1-f<1.$$  The same bounds also hold for $1-\Pi f$ and  $1-f_{\infty}.$ In particular, $\int_{B}(1-\Pi f)(1-f_{\infty})e^{-\mathcal{E}}dk$ is  bounded from below and above by positive constants. These bounds and \eqref{r-r} show that   \eqref{rr} holds. 

Next, \eqref{r-m} shows that $\rho-\rho_{\infty}$ and $e^{\mu}-e^{\mu_{\infty}}$ have the same sign. Since  $e^{\mu}-e^{\mu_{\infty}}$ and $\mu-\mu_{\infty}$ have the same sign, the mean value theorem and \eqref{b1} provide
$$(\rho-\rho_{\infty})(e^{\mu}-e^{\mu_{\infty}})\leq e^{\mu_2}(\rho-\rho_{\infty})({\mu}-{\mu_{\infty}}).$$
The last two estimates show
$$(\rho-\rho_{\infty})({\mu}-{\mu_{\infty}})\geq \frac{(\Pi f-f_{\infty})^2 e^{2\mathcal{E}} e^{\min_{\mathbb{T}^d}\{V\}-\mu_{2}}}{(1-\Pi f)^2(1-f_{\infty})^2}\int_{B}(1-\Pi f)(1-f_{\infty})e^{-\mathcal{E}}dk.$$
Since $1-\Pi f,$  $1-f_{\infty},$ and $\int_{B}(1-\Pi f)(1-f_{\infty})e^{-\mathcal{E}}dk$ are bounded from below and above by positive constants, \eqref{rk} also holds.
\end{proof}

\begin{lemma} Let  $\mathcal{E}:B\to \mathbb{R},$  $V:\mathbb{T}^d\to \mathbb{R}$, and $\sigma :\mathbb{T}^d\times B\times B\to \mathbb{R}$ be smooth functions.  Assume there is a constant $\sigma_{-}>0$ such that
\begin{equation*}
   \sigma_{-} \leq \sigma(x,k,k'), \, \, \forall\, x\in \mathbb{T}^d,\,\,\, \forall\, k,k'\in B.
\end{equation*}
 Then, for functions $f$ satisfying \eqref{b1}, there exists a constant  $C_5>0$  such that
\begin{align}\label{C_5}
\mathcal{D}(f)\geq C_5 \int_{\mathbb{T}^d}\int_{{B}}(f-\Pi f)^2 e^{\mathcal{E}}dkdx.
\end{align}
\end{lemma}
\begin{proof}
The mean value theorem implies that there exists $\theta\in [0,1]$ such that 
$$(F-F')(\log F-\log F')=\frac{(F-F')^2}{\theta F+(1-\theta)F'}.$$
By \eqref{b1} we have $ \theta F+(1-\theta)F'\leq e^{\mu_2}$ and so 
$$(F-F')(\log F-\log F') \geq e^{-\mu_2}(F-F')^2.$$
This estimate and the lower bound on $\sigma$ imply 
\begin{align}\label{sigma/2}
\mathcal{D}(f)\geq \frac{\sigma_{-}e^{-\mu_2}}{2}\int_{\mathbb{T}^d}\int_{ B}\int_{ B}e^{-\mathcal{E}-\mathcal{E}'-V}(1-f)(1-f')(F-F')^2 dk'dkdx.
\end{align}
We have
 \begin{align*}
 (F-F')^2&=(F-e^{\mu}+e^{\mu}-F')^2\\
 &=\left(\frac{f-\Pi f}{e^{-\mathcal{E}-V}(1-f)(1-\Pi f)}\right)^2+\left(\frac{f'-\Pi f'}{e^{-\mathcal{E}'-V}(1-f')(1-\Pi f')}\right)^2\\
 &\,\,\,\,-\frac{2(f-\Pi f)(f'-\Pi f')}{e^{-\mathcal{E}-\mathcal{E}'-2V}(1-f)(1-f')(1-\Pi f)(1-\Pi f')}.
 \end{align*}
 This equation let us write
\begin{align*}
&\int_{\mathbb{T}^d}\int_{ B}\int_{ B}e^{-\mathcal{E}-\mathcal{E}'-V}(1-f)(1-f')(F-F')^2 dk'dkdx\\
&\geq \int_{\mathbb{T}^d}\int_{ B}\int_{ B}e^{-\mathcal{E}-\mathcal{E}'-V}(1-f)(1-f')(1-\Pi f)(1-\Pi f')(F-F')^2 dk'dkdx\\
&= \int_{\mathbb{T}^d}\int_{ B}\int_{ B}e^{\mathcal{E}-\mathcal{E}'+V}(1-f')(1-\Pi f') \frac{(f-\Pi f)^2}{(1-f)(1-\Pi f)}dk'dkdx\\
&\,\,\,\,+\int_{\mathbb{T}^d}\int_{ B}\int_{ B}e^{-\mathcal{E}+\mathcal{E}'+V}(1-f)(1-\Pi f) \frac{(f'-\Pi f')^2}{(1-f')(1-\Pi f')}dk'dkdx\\
&\,\,\,\,-2\int_{\mathbb{T}^d}e^{V}\int_{ B}\int_{ B}(f-\Pi f)(f'-\Pi f')dk'dkdx\\
&=2\int_{\mathbb{T}^d}\left(\int_{ B}e^{-\mathcal{E}'}(1-f')(1-\Pi f')dk'\right)\left( \int_{ B} \frac{(f-\Pi f )^2}{(1-f)(1-\Pi f)}e^{\mathcal{E}+V}dk\right)dx\\
&\,\,\,\,-2\int_{\mathbb{T}^d}\left(\int_{ B}(f-\Pi f)dk\right)^2dx.
\end{align*}
According to the definition of $\Pi f,$ the last integral is zero. Moreover, because of the bound \eqref{b1} we can estimate 
$$0< \frac{e^{\min_B \{\mathcal{E}\}+\min_{\mathbb{T}^d} \{V\}-\mu_2}}{1+e^{\min_B \{\mathcal{E}\}+\min_{\mathbb{T}^d} \{V\}-\mu_2}}\leq 1-f<1.$$  The same bounds also hold for $\Pi f.$ Hence, we have
\begin{align*}
&\int_{\mathbb{T}^d}\int_{ B}\int_{ B}e^{-\mathcal{E}-\mathcal{E}'-V}(1-f)(1-f')(F-F')^2 dk'dkdx\\
&\geq  \frac{2e^{2\min_B \{\mathcal{E}\}+2\min_{\mathbb{T}^d} \{V\}-2\mu_2}}{(1+e^{\min_B \{\mathcal{E}\}+\min_{\mathbb{T}^d} \{V\}-\mu_2})^2}\left(\int_B e^{-\mathcal{E}'}dk'\right)\int_{\mathbb{T}^d} \int_{ B} (f-\Pi f)^2e^{\mathcal{E}+V}dxdk.
\end{align*}
This equation and \eqref{sigma/2} imply
\begin{align*}
\mathcal{D}(f) \geq\frac{\sigma_- e^{2\min_B \{\mathcal{E}\}+3\min_{\mathbb{T}^d} \{V\}-3\mu_2}}{(1+e^{\min_B \{\mathcal{E}\}+\min_{\mathbb{T}^d} \{V\}-\mu_2})^2}\left(\int_B e^{-\mathcal{E}'}dk'\right)\int_{\mathbb{T}^d} \int_{ B} (f-\Pi f)^2e^{\mathcal{E}}dxdk.
\end{align*}
Hence, \eqref{C_5} holds with $C_5\colonequals \frac{\sigma_{-} e^{2\min_B \{\mathcal{E}\}+3\min_{\mathbb{T}^d} \{V\}-3\mu_2}}{(1+e^{\min_B \{\mathcal{E}\}+\min_{\mathbb{T}^d} \{V\}-\mu_2})^2}\int_B e^{-\mathcal{E}(k')}dk'. $
\end{proof}

\section{Lyapunov functional}
We  recall 
$$\rho_{\infty}(x)\colonequals \int_{B}f_{\infty}(x, k)dk$$
and 
$$\rho(t,x)\colonequals \int_{B}f(t,x,k)dk.$$
We define a vector function
\begin{align*}
J(t,x) \colonequals \int_B\nabla_k\mathcal{E}(k)\log\left(\frac{f(t,x,k)}{1-f(t,x,k)}\right) e^{-\mathcal{E}(k)}dk\in \mathbb{R}^d
\end{align*}
and a constant matrix 
$$A\colonequals \int_B \frac{\partial^2\mathcal{E}(k)}{\partial k^2}e^{-\mathcal{E}(k)} dk\in \mathbb{R}^{d\times d}.$$
We consider the elliptic equation
\begin{equation*}
-\mathrm{div}_x(A\nabla_x \phi(t,x))=\rho(t,x)-\rho_{\infty},\, \, \, \int_{\mathbb{T}^d} \phi \,dx=0, \, \, \, x\in \mathbb{T}^d, \, \, t\geq 0.
\end{equation*}
Since $\rho(t,\cdot)-\rho_{\infty}$ has zero mass for every $t\geq 0$ and  $A$ is positive definite,  the classical elliptic regularity results (see, for example, \cite{Ell}) show that there exists a unique solution $\phi$ satisfying 
\begin{equation}\label{C_R}
||\phi(t,\cdot)||_{H^2(\mathbb{T}^d)}\leq C_R ||\rho(t,\cdot)-\rho_{\infty}||_{L^2(\mathbb{T}^d)},\,\,\,\,\,\forall\,t\geq 0,
\end{equation}
where the constant $C_R>0$  is independent of $t$.
 Then we define a modified entropy functional
\begin{equation*}
\mathrm{E}(f)\colonequals \mathrm{H}(f|f_{\infty})+\delta \int_{\mathbb{T}^d} \nabla_x \phi \cdot J dx,
\end{equation*}
 where $\delta>0$  will be chosen later. 
\begin{lemma}\label{lem:eq} 
Let  $\mathcal{E}:B\to \mathbb{R}$ and  $V:\mathbb{T}^d\to \mathbb{R}$ be smooth functions.
 If $\delta>0$ is  small enough, for functions $f$ satisfying \eqref{b1}, there are constants ${C}_6>0$ and $\mathcal{C}_7>0$ such that \begin{equation}\label{E equiv}
 {C}_6\int_{\mathbb{T}^d}\int_B(f-f_{\infty})^2e^{\mathcal{E}}dkdx
 \leq \mathrm{E}(f)\leq {C}_7 \int_{\mathbb{T}^d}\int_B(f-f_{\infty})e^{\mathcal{E}}dkdx.
\end{equation}
\end{lemma}
\begin{proof}
The function 
\begin{align*}
f\mapsto f\log\frac{f}{f_{\infty}}+(1-f)\log\frac{1-f}{1-f_{\infty}}
\end{align*}
and its derivative with respect to $f$  vanish at $f=f_{\infty}$. This fact and  the mean  value theorem provide
\begin{align*}
f\log\frac{f}{f_{\infty}}+(1-f)\log\frac{1-f}{1-f_{\infty}}=\frac{(f-f_{\infty})^2}{2g(1-g)},
\end{align*}
where the function $g$ can be written as  $g\colonequals \theta f+(1-\theta)f_{\infty}$  for some function $\theta(t,x,k)\in [0,1].$ In particular, $g$  satisfies \eqref{b1}. As a consequence $\frac{e^{-\mathcal{E}}}{g(1-g)}$ is  bounded from below and above by positive constants. This implies $\mathrm{H}(f|f_{\infty})$ is equivalent to 
$\int_{\mathbb{T}^d}\int_B(f-f_{\infty})^2e^{\mathcal{E}}dkdx.$

 Next, we estimate  $\left|\int_{\mathbb{T}^d} \nabla_x \phi\cdot J dx\right|$ using the H\"older inequality
$$ \left|\int_{\mathbb{T}^d} \nabla_x \phi\cdot J dx\right|\leq ||\nabla_x \phi||_{L^2(\mathbb{T}^d)}||J||_{L^2(\mathbb{T}^d)}.$$
\eqref{C_R} shows that  $||\nabla_x \phi||_{L^2(\mathbb{T}^d)}$ is bounded (up to a constant) by $||\rho-\rho_{\infty}||_{L^2(\mathbb{T}^d)}.$ Moreover, the H\"older inequality provides $$||\rho-\rho_{\infty}||_{L^2(\mathbb{T}^d)}\leq \sqrt{\int_Be^{-\mathcal{E}}dk}\sqrt{\int_{\mathbb{T}^d}\int_B(f-f_{\infty})^2e^{\mathcal{E}}dkdx}.$$  
Integrating by parts one can show  $$\displaystyle \int_{B }\nabla_k\mathcal{E}\log\frac{f_{\infty}}{1-f_{\infty}}\,e^{-\mathcal{E}}dk=-\int_{B }\nabla_k\mathcal{E}(\mu_{\infty}-\mathcal{E}-V)e^{-\mathcal{E}}dk=0.$$ This let us write  
\begin{align*}
||J(t)||_{L^2(\mathbb{T}^d)}&=\sqrt{\int_{\mathbb{T}^d}\left|\int_{B}\nabla_k\mathcal{E}\left(\log\frac{f}{1-f}-\log\frac{f_{\infty}}{1-f_{\infty}}\right)e^{-\mathcal{E}}dk \right|^2dx}.
\end{align*}
By the mean value theorem there is a function $\tilde{\theta}(t,x,k)\in [0,1]$ such that for $h=\tilde{\theta} f+(1-\tilde{\theta})f_{\infty}$
$$\log\frac{f}{1-f}-\log\frac{f_{\infty}}{1-f_{\infty}}=\frac{f-f_{\infty}}{h(1-h)}.$$
  \eqref{b1} implies $\frac{e^{-\mathcal{E}}}{h(1-h)}$ is bounded from below and above by  positive constants.  Hence, there is a positive constant $C>0$ such that 
\begin{align*}
||J(t)||_{L^2(\mathbb{T}^d)}&\leq C\sqrt{\int_{\mathbb{T}^d}\left|\int_{B}\nabla_k\mathcal{E}(f-f_{\infty})dk \right|^2dx}.
\end{align*} 
The H\"older inequality  implies
\begin{align*}
||J(t)||_{L^2(\mathbb{T}^d)}&\leq C \sqrt{\int_{\mathbb{T}^d}\left(\int_{B}|\nabla_k \mathcal{E}|^2e^{-\mathcal{E}}dk \right)\left(\int_{B}(f-f_{\infty})^2e^{\mathcal{E}}dk\right) dx}\\
&\leq C \sqrt{\int_{B}|\nabla_k \mathcal{E}|^2e^{-\mathcal{E}} dk}\sqrt{\int_{\mathbb{T}^d}\int_{B} (f-f_{\infty})^2e^{\mathcal{E}}dk dx}. 
\end{align*}
These estimates show that there is a constant $\tilde{C}>0$ such that
$$ \left|\int_{\mathbb{T}^d} \nabla_x \phi\cdot J dx\right|\leq \tilde{C} \int_{\mathbb{T}^d}\int_{B} (f-f_{\infty})^2e^{\mathcal{E}}dk dx.$$
Therefore, we have
$$\mathrm{H}(f|f_{\infty})-\delta \tilde{C}\int_{\mathbb{T}^d}\int_{B} (f-f_{\infty})^2e^{\mathcal{E}}dk dx\leq \mathrm{E}(f)\leq \mathrm{H}(f|f_{\infty})+\delta \tilde{C}\int_{\mathbb{T}^d}\int_{B} (f-f_{\infty})^2e^{\mathcal{E}}dk dx.$$
We have proved that 
$\mathrm{H}(f|f_{\infty})$ is equivalent to $\int_{\mathbb{T}^d}\int_{B} (f-f_{\infty})^2e^{\mathcal{E}}dk dx.$ Hence, if $\delta$ is small enough, the claimed estimate  \eqref{E equiv} holds.
\end{proof}
 
\bigskip

\begin{proof}[\textbf{Proof of Theorem \ref{main th}}] Theorem \ref{existence} provides the existence of a unique distributional solution satisfying  \eqref{b1}. Hence, it remains to prove \eqref{lamda}.
We consider the modified entropy functional \begin{equation*}
\mathrm{E}(f(t))\colonequals \mathrm{H}(f(t)|f_{\infty})+\delta \int_{\mathbb{T}^d} \nabla_x \phi\cdot Jdx.
\end{equation*}
The time derivative of this functional equals
\begin{align}\label{dt E}
\frac{d}{dt}\mathrm{E}(f(t))=\frac{d}{dt}\mathrm{H}(f(t)|f_{\infty})+\delta \int_{\mathbb{T}^d} \partial_t\nabla_x \phi\cdot J dx+\delta \int_{\mathbb{T}^d} \nabla_x \phi\cdot \partial_t J dx
\end{align}
We first estimate this time derivative by assuming that the solution $f$ is in $W^{1,\infty}([0,T]\times \mathbb{T}^d\times B)$ for all $T>0.$\\
\textbf{Step 1: Estimate for $\displaystyle \int_{\mathbb{T}^d} \partial_t\nabla_x \phi\cdot J dx.$ }\\
 The H\"older inequality provides
 \begin{equation}\label{Hol}
 \int_{\mathbb{T}^d} \partial_t\nabla_x \phi\cdot J dx\leq ||\partial_t\nabla_x \phi||_{L^2(\mathbb{T}^d)} ||J||_{L^2(\mathbb{T}^d)}.
  \end{equation}
First, we estimate $||\partial_t\nabla_x \phi||_{L^2(\mathbb{T}^d)}$: 
Integrating \eqref{Eq} with respect to $k$ and using $ \int_{B} Q(f)dk=0,$ we obtain
\begin{equation*}
\partial_t \rho+\mathrm{div}_x\left(\int_B\nabla_k \mathcal{E}fdk\right)=0.
\end{equation*}
This shows  $$-\mathrm{div}_x(A\partial_t \nabla_x\phi)=\partial_t \rho=-\mathrm{div}_x\left(\int_B\nabla_k \mathcal{E}fdk\right).$$
We multiply this equation by $\partial_t \phi$ and  integrate by parts
\begin{align*}
\int_{\mathbb{T}^d}(\partial_t \nabla_x \phi(t))^TA(\partial_t \nabla_x \phi(t)) dx=&-\int_{\mathbb{T}^d}\partial_t  \phi\, \mathrm{div}_x\left(\int_B\nabla_k \mathcal{E}fdk\right)dx\\=&\int_{\mathbb{T}^d}\partial_t  \nabla_x \phi \cdot \left(\int_B\nabla_k \mathcal{E}fdk\right) dx\\
\leq & ||\partial_t \nabla_x \phi(t)||_{L^2(\mathbb{T}^d)} \sqrt{\int_{\mathbb{T}^d}\left|\int_B\nabla_k \mathcal{E}fdk\right|^2dx}. 
\end{align*}
Since the matrix $A$ is positive definite, we have
$$||\partial_t \nabla_x \phi(t)||_{L^2(\mathbb{T}^d)} \leq \alpha^{-1}_0 \sqrt{\int_{\mathbb{T}^d}\left|\int_B\nabla_k \mathcal{E}fdk\right|^2dx},$$ 
where $\alpha_0>0$ is the smallest eigenvalue of $A$.
 Integrating by parts one can show 
  $$\int_B\nabla_k \mathcal{E}\Pi f dk=0.$$
   We use this equation and  the H\"older inequality 
\begin{align*}\int_{\mathbb{T}^d}\left|\int_B\nabla_k \mathcal{E}fdk\right|^2dx&=\int_{\mathbb{T}^d}\left|\int_B\nabla_k \mathcal{E}(f-\Pi f)dk\right|^2dx\\
&\leq \left(\int_B|\nabla_k \mathcal{E}|^2e^{- \mathcal{E}}dk\right)\int_{\mathbb{T}^d}\int_B(f-\Pi f)^2e^{\mathcal{E}}dkdx. 
\end{align*}
The last two estimates show
\begin{equation}\label{Hol1}
||\partial_t \nabla_x \phi(t)||_{L^2(\mathbb{T}^d)} \leq \alpha^{-1}_0\sqrt{\int_B|\nabla_k \mathcal{E}|^2e^{- \mathcal{E}}dk}\sqrt{\int_{\mathbb{T}^d}\int_B(f-\Pi f)^2e^{\mathcal{E}}dkdx}.
\end{equation}
Next, we  estimate $||J||_{L^2(\mathbb{T}^d)}:$ Integrating by parts one can show
 $$\int_{B }\nabla_k\mathcal{E}\log\frac{\Pi f}{1-\Pi f}\,e^{-\mathcal{E}}dk=-\int_{B }\nabla_k\mathcal{E}(\mu-\mathcal{E}-V){e^{-\mathcal{E}}}dk=0.$$ This let us  write  
\begin{align*}
||J(t)||_{L^2(\mathbb{T}^d)}&=\sqrt{\int_{\mathbb{T}^d}\left|\int_{B}\nabla_k\mathcal{E}\left(\log\frac{f}{1-f}-\log\frac{\Pi f}{1-\Pi f}\right){e^{-\mathcal{E}}}dk \right|^2dx}.
\end{align*}
By the mean value theorem we have 
$$\log\frac{f}{1-f}-\log\frac{\Pi f}{1-\Pi f}=\frac{f-\Pi f}{g(1-g)},$$
where the function $g$ can be written as $g=\theta f+(1-\theta)f_{\infty}$ for some function  $\theta(t,x,k)\in [0,1].$ In particular, $g$ satisfies  \eqref{b1} and so $\frac{e^{-\mathcal{E}}}{g(1-g)}$ is bounded from below and above by  positive constants.  Hence, there is a positive constant $C>0$ such that 
\begin{equation}\label{logC}
\left|\log\frac{f}{1-f}-\log\frac{\Pi f}{1-\Pi f}\right|\leq C{e^{\mathcal{E}}}|f-\Pi f|.
\end{equation}
This helps to estimate
\begin{align*}
||J(t)||_{L^2(\mathbb{T}^d)}&\leq C\sqrt{\int_{\mathbb{T}^d}\left(\int_{B}|\nabla_k\mathcal{E}||f-\Pi f|dk \right)^2dx}.
\end{align*} 
The H\"older inequality  implies
\begin{align*}
||J(t)||_{L^2(\mathbb{T}^d)}&\leq C \sqrt{\int_{\mathbb{T}^d}\left(\int_{B}|\nabla_k \mathcal{E}|^2e^{-\mathcal{E}}dk \right)\left(\int_{B}(f-\Pi f)^2e^{\mathcal{E}}dk\right) dx}\\
&\leq C \sqrt{\int_{B}|\nabla_k \mathcal{E}|^2e^{-\mathcal{E}} dk}\sqrt{\int_{\mathbb{T}^d}\int_{B} (f-\Pi f)^2e^{\mathcal{E}}dk dx}. 
\end{align*} 
 \eqref{Hol}, \eqref{Hol1}, and the above estimate yield
\begin{equation}\label{tphi j}\int_{\mathbb{T}^d} \partial_t\nabla_x \phi\cdot J dx\leq \tilde{C}\int_{\mathbb{T}^d}\int_{B} (f-\Pi f)^2e^{\mathcal{E}}dk dx
\end{equation}
with $\tilde{C}\colonequals  C\alpha^{-1}_0\int_B|\nabla_k \mathcal{E}|^2e^{- \mathcal{E}}dk.$\\

\textbf{Step 2: Estimate for $\displaystyle \int_{\mathbb{T}^d} \nabla_x \phi \cdot  \partial_t J dx.$}\\
We compute  
 \begin{align*}
\partial_t J=&\int_B  \nabla_k \mathcal{E} \frac{\partial_t f}{f(1-f)} e^{-\mathcal{E}}dk\\
=&-\int_{B} \nabla_k\mathcal{E}\left(\nabla_k\mathcal{E}\cdot \frac{\nabla_x f}{f(1-f)}-\nabla_x V\cdot \frac{\nabla_k f}{f(1-f)}\right)e^{-\mathcal{E}}dk+\int_B\nabla_k\mathcal{E} \frac{Q(f)}{f(1-f)}e^{-\mathcal{E}}dk\\
=& -\int_{B} \nabla_k\mathcal{E}\left(\nabla_k\mathcal{E}\cdot \nabla_x\log\frac{ f}{(1-f)}-\nabla_x V\cdot \nabla_k\log\frac{ f}{(1-f)}\right)e^{-\mathcal{E}}dk\\
&+\int_B\nabla_k\mathcal{E} \frac{Q(f)}{f(1-f)}e^{-\mathcal{E}}dk.
\end{align*}
It is easy to check  $$\nabla_k\mathcal{E}\cdot \nabla_x\log\frac{ \Pi f}{(1-\Pi f)}-\nabla_x V\cdot \nabla_k\log\frac{ \Pi f}{(1-\Pi f)}=\nabla_k\mathcal{E}\cdot \nabla_x \mu.$$
Using this equation we can write  
\begin{align*}
-\int_{B}& \nabla_k\mathcal{E}\left(\nabla_k\mathcal{E}\cdot \nabla_x\log\frac{ f}{(1-f)}-\nabla_x V\cdot \nabla_k\log\frac{ f}{(1-f)}\right)e^{-\mathcal{E}}dk\\=&-\int_{B} \nabla_k\mathcal{E}\left(\nabla_k\mathcal{E}\cdot \nabla_x\log\frac{ \Pi f}{(1-\Pi f)}-\nabla_x V\cdot \nabla_k\log\frac{ \Pi f}{(1-\Pi f)}\right)e^{-\mathcal{E}}dk\\
&-\int_{B} \nabla_k\mathcal{E}\left(\nabla_k\mathcal{E}\cdot \nabla_x\log\frac{ f(1-\Pi f)}{(1-f)\Pi f}-\nabla_x V\cdot \nabla_k\log\frac{ f(1-\Pi f)}{(1-f)\Pi f}\right)e^{-\mathcal{E}}dk\\
=& -\left(\int_{B} \nabla_k\mathcal{E}\otimes \nabla_k\mathcal{E} e^{-\mathcal{E}}dk\right) \nabla_x \mu-\int_{B} \nabla_k\mathcal{E}\nabla_k\mathcal{E}\cdot \nabla_x\log\frac{ f(1-\Pi f)}{(1-f)\Pi f}\,e^{-\mathcal{E}}dk\\
&-\int_B \left(\frac{\partial^2 \mathcal{E}}{\partial k^2}-\nabla_k\mathcal{E}\otimes \nabla_k\mathcal{E}\right)\nabla_x V\log\frac{ f(1-\Pi f)}{(1-f)\Pi f}\,e^{-\mathcal{E}}dk.
\end{align*}
Note that  $ \int_{B} \nabla_k\mathcal{E}\otimes \nabla_k\mathcal{E} e^{-\mathcal{E}}dk=\int_{B} \frac{\partial^2\mathcal{E}}{\partial k^2} e^{-\mathcal{E}}dk\in \mathbb{R}^{d\times d}$.  We assume that $A\colonequals \int_{B} \frac{\partial^2\mathcal{E}}{\partial k^2} e^{-\mathcal{E}}dk$ is positive definite.
The above equations let us write 
\begin{align}\label{phi.tJ}
\int_{\mathbb{T}^d} \nabla_x \phi\cdot  \partial_t J dx=&\int_{\mathbb{T}^d}\mathrm{div}_x (A\nabla_x \phi)(\mu-\mu_{\infty}) dx\nonumber\\
&+\int_{\mathbb{T}^d}\int_{B} \nabla_k^T\mathcal{E}\frac{\partial^2\phi}{\partial x^2}\nabla_k\mathcal{E}\log\frac{ f(1-\Pi f)}{(1-f)\Pi f}\,e^{-\mathcal{E}}dkdx\nonumber \\
&-\int_{\mathbb{T}^d} \int_B \nabla_x^T\phi\left(\frac{\partial^2 \mathcal{E}}{\partial k^2}-\nabla_k\mathcal{E}\otimes \nabla_k\mathcal{E}\right)\nabla_x V\log\frac{ f(1-\Pi f)}{(1-f)\Pi f}\,e^{-\mathcal{E}}dkdx\nonumber \\
&+\int_{\mathbb{T}^d}\int_B\nabla_x \phi\cdot\nabla_k\mathcal{E} \frac{Q(f)}{f(1-f)}e^{-\mathcal{E}}dkdx.
\end{align}
 We now estimate the terms of \eqref{phi.tJ}: The first term  can be estimated as 
 \begin{align}\label{first}
 \int_{\mathbb{T}^d}\mathrm{div}_x (A\nabla_x \phi)(\mu-\mu_{\infty}) dx=&-\int_{\mathbb{T}^d}(\rho-\rho_{\infty})(\mu-\mu_{\infty})dx\nonumber\\=&-\frac{1}{\int_B e^{-\mathcal{E}(k')}dk'}\int_{\mathbb{T}^d}\int_B\rho-\rho_{\infty})(\mu-\mu_{\infty})e^{-\mathcal{E}}dkdx\nonumber \\
 \leq &-C'\int_{\mathbb{T}^d}\int_B(\Pi f-f_{\infty})^2e^{\mathcal{E}}dkdx,
 \end{align}
 where the last inequality follows from \eqref{rk} and $C'\colonequals \frac{C_4}{\int_B e^{-\mathcal{E}(k')}dk'}$.

 To estimate the second term of \eqref{phi.tJ}, we use \eqref{logC} and the H\"older inequality 
 \begin{align*}
 \int_{\mathbb{T}^d}\int_{B} & \nabla_k^T\mathcal{E}\frac{\partial^2\phi}{\partial x^2}\nabla_k\mathcal{E}\log\frac{ f(1-\Pi f)}{(1-f)\Pi f}\,e^{-\mathcal{E}}dkdx\\
 & \leq C \int_{\mathbb{T}^d}\left|\left|\frac{\partial^2\phi}{\partial x^2}\right|\right|_F \left(\int_B |\nabla_k \mathcal{E}|^2   |f-\Pi f|dk\right)dx
 \\
 & \leq C \left|\left|\frac{\partial^2 \phi}{\partial x^2}\right|\right|_{L^2(\mathbb{T}^d)}\sqrt{\int_{B}|\nabla_k \mathcal{E}|^4 e^{-\mathcal{E}}dk}
 \sqrt{ \int_{\mathbb{T}^d}\int_{B} (f-\Pi f)^2e^{\mathcal{E}} dkdx}.
 \end{align*}
 \eqref{C_R} shows that  $\left|\left|\frac{\partial^2 \phi}{\partial x^2}\right|\right|_{L^2(\mathbb{T}^d)}$ is bounded (up to a constant) by $||\rho-\rho_{\infty}||_{L^2(\mathbb{R}^d)}$. Moreover, \eqref{rr} implies that   $||\rho-\rho_{\infty}||_{L^2(\mathbb{T}^d)}$ is bounded (up to a constant) by $\sqrt{ \int_{\mathbb{T}^d}\int_{B} (\Pi f-f_{\infty})^2e^{\mathcal{E}} dkdx}.$ Hence, there is a constant ${C}''>0$ such that 
 \begin{align}\label{second}
  \int_{\mathbb{T}^d}\int_{B}&  \nabla_k^T\mathcal{E}\frac{\partial^2\phi}{\partial x^2}\nabla_k\mathcal{E}\log\frac{ f(1-\Pi f)}{(1-f)\Pi f}\,e^{-\mathcal{E}}dkdx\nonumber \\
  &\leq {C}'' \sqrt{\int_{\mathbb{T}^d}\int_{B} (\Pi f-f_{\infty} )^2e^{\mathcal{E}} dkdx} \sqrt{\int_{\mathbb{T}^d}\int_{B} (f-\Pi f)^2e^{\mathcal{E}} dkdx}.
 \end{align}
To estimate the third term in \eqref{phi.tJ}, we use \eqref{logC} and   the H\"older inequality 
 \begin{align*}
 -&\int_{\mathbb{T}^d} \int_B  \nabla_x^T\phi\left(\frac{\partial^2 \mathcal{E}}{\partial k^2}-\nabla_k\mathcal{E}\otimes \nabla_k\mathcal{E}\right)\nabla_x V\log\frac{ f(1-\Pi f)}{(1-f)\Pi f}\,e^{-\mathcal{E}}dkdx\\
 &\leq C \max_{\mathbb{T}^d}\{|\nabla_x V|\}\int_{\mathbb{T}^d}|\nabla_x \phi|\left(\int_B \left|\left|\frac{\partial^2 \mathcal{E}}{\partial k^2}-\nabla_k\mathcal{E}\otimes \nabla_k\mathcal{E}\right|\right|_F|f-\Pi f|dk\right)dx\\
 &\leq C \max_{\mathbb{T}^d}\{|\nabla_x V|\}||\nabla_x \phi||_{L^2(\mathbb{T}^d)}\sqrt{\int_B \left|\left|\frac{\partial^2 \mathcal{E}}{\partial k^2}-\nabla_k\mathcal{E}\otimes \nabla_k\mathcal{E}\right|\right|_F^2e^{-\mathcal{E}}dk}\sqrt{\int_{\mathbb{T}^d}\int_B (f-\Pi f)^2e^{\mathcal{E}}dkdx}.
 \end{align*}
  \eqref{C_R} shows that  $||\nabla_x \phi||_{L^2(\mathbb{T}^d)}$ is bounded (up to a constant) by $||\rho-\rho_{\infty}||_{L^2(\mathbb{R}^d)}$. Moreover, \eqref{rr} implies that   $||\rho-\rho_{\infty}||_{L^2(\mathbb{T}^d)}$ is bounded (up to a constant) by $\sqrt{ \int_{\mathbb{T}^d}\int_{B} (\Pi f-f_{\infty})^2e^{\mathcal{E}} dkdx}.$  Hence, there is a constant $\tilde{C}'>0$ such that  
 \begin{align}\label{third}
 -&\int_{\mathbb{T}^d} \int_B  \nabla_x^T\phi\left(\frac{\partial^2 \mathcal{E}}{\partial k^2}-\nabla_k\mathcal{E}\otimes \nabla_k\mathcal{E}\right)\nabla_x V\log\frac{ f(1-\Pi f)}{(1-f)\Pi f}\,e^{-\mathcal{E}}dkdx\nonumber \\
 &\leq \tilde{C}'\sqrt{\int_{\mathbb{T}^d}\int_B (\Pi f-f_{\infty})^2e^{\mathcal{E}}dkdx} \sqrt{\int_{\mathbb{T}^d}\int_B (f-\Pi f)^2e^{\mathcal{E}}dkdx}
 \end{align}

We now estimate  the last term of \eqref{phi.tJ}:  \eqref{b1} implies $\frac{e^{-\mathcal{E}}}{f(1-f)}$ is bounded from above by a positive constant. Hence, there is a positive constant $\tilde{ C}''$ such that 
 \begin{align*}
 \int_{\mathbb{T}^d}\int_B\nabla_x \phi\cdot\nabla_k\mathcal{E} \frac{Q(f)}{f(1-f)}e^{-\mathcal{E}}dkdx  \leq \tilde{ C}'' \int_{\mathbb{T}^d}|\nabla_x \phi|\left|\int_B\nabla_k\mathcal{E}Q(f)dk\right| dx.
 \end{align*}
 We use  the H\"older inequality and \eqref{Q<}
  \begin{align*} \int_{\mathbb{T}^d}\int_B & \nabla_x \phi\cdot\nabla_k \mathcal{E} \frac{Q(f)}{f(1-f)}e^{-\mathcal{E}}dkdx \\
  &\leq \tilde{ C}''||\nabla_x \phi||_{L^2(\mathbb{T}^d)} \sqrt{\int_{B}|\nabla_k \mathcal{E}(k')|^2 e^{-\mathcal{E}(k')}dk'} 
  \sqrt{\int_{\mathbb{T}^d}\int_B |Q(f)|^2e^{\mathcal{E}(k)}dkdx}
   \\
  &\leq \tilde{ C}''\sqrt{C_1}||\nabla_x \phi||_{L^2(\mathbb{T}^d)} \sqrt{\int_{B}|\nabla_k \mathcal{E}(k')|^2 e^{-\mathcal{E}(k')}dk'} 
 \sqrt{\mathcal{D}(f)}.
 \end{align*}
 Similarly,  $||\nabla_x \phi||_{L^2(\mathbb{T}^d)}$ is bounded (up to a constant) by $||\rho-\rho_{\infty}||_{L^2(\mathbb{R}^d)}$. Moreover, \eqref{rr} implies that   $||\rho-\rho_{\infty}||_{L^2(\mathbb{T}^d)}$ is bounded (up to a constant) by $\sqrt{ \int_{\mathbb{T}^d}\int_{B} (\Pi f-f_{\infty})^2e^{\mathcal{E}} dkdx}.$  Hence, there is a constant $\bar{C}>0$ such that 
 \begin{equation}\label{tj3}
 \int_{\mathbb{T}^d}\int_B  \nabla_x \phi\cdot\nabla_k \mathcal{E} \frac{Q(f)}{f(1-f)}e^{-\mathcal{E}}dkdx \leq \bar{C} \sqrt{\int_{\mathbb{T}^d}\int_{B} (\Pi f-f_{\infty} )^2e^{\mathcal{E}} dkdx} \sqrt{\mathcal{D}(f)} .
 \end{equation}
 \eqref{phi.tJ}, \eqref{first}, \eqref{second}, \eqref{third}, and \eqref{tj3} imply 
 \begin{align}\label{phitj}\int_{\mathbb{T}^d} \nabla_x \phi\cdot \partial_t J dx \leq & -C'\int_{\mathbb{T}^d}\int_B(\Pi f-f_{\infty})^2e^{\mathcal{E}}dkdx\nonumber\\
 &+(C''+\tilde{C}')\sqrt{\int_{\mathbb{T}^d}\int_B (\Pi f-f_{\infty})^2e^{\mathcal{E}}dkdx} \sqrt{\int_{\mathbb{T}^d}\int_B (f-\Pi f)^2e^{\mathcal{E}}dkdx}\nonumber\\
 &+\bar{C} \sqrt{\int_{\mathbb{T}^d}\int_{B} (\Pi f-f_{\infty} )^2e^{\mathcal{E}} dkdx} \sqrt{\mathcal{D}(f)}.
 \end{align}\\
 \textbf{Step 3: Gr\"onwall's inequality. }\\
 We now summarize all estimates. \eqref{dt E}, \eqref{dtH}, \eqref{C_5}, \eqref{tphi j}, and \eqref{phitj} yield
 \begin{align*}
 \frac{d}{dt}\mathrm{E}(f(t))\leq & -\frac{\mathcal{D}(f)}{2}-\left(\frac{C_5}{2}-\delta \tilde{C}\right)\int_{\mathbb{T}^d}\int_{B} (f-\Pi f )^2e^{\mathcal{E}} dkdx\\
  & -\delta C'\int_{\mathbb{T}^d}\int_B(\Pi f-f_{\infty})^2e^{\mathcal{E}}dkdx\nonumber\\
 &+\delta (C''+\tilde{C}')\sqrt{\int_B (\Pi f-f_{\infty})^2e^{\mathcal{E}}dkdx} \sqrt{\int_B (f-\Pi f)^2e^{\mathcal{E}}dkdx}\nonumber\\
 &+\delta\bar{C} \sqrt{\int_{\mathbb{T}^d}\int_{B} (\Pi f-f_{\infty} )^2e^{\mathcal{E}} dkdx} \sqrt{\mathcal{D}(f)}.
 \end{align*}
 One can consider the right hand side of this equation as a quadratic form with respect to $\sqrt{\mathcal{D}(f)}, $ $\sqrt{\int_{\mathbb{T}^d}\int_{B} (\Pi f-f_{\infty} )^2e^{\mathcal{E}} dkdx}$, and  $\sqrt{\int_{\mathbb{T}^d}\int_{B} (\Pi f-f_{\infty} )^2e^{\mathcal{E}} dkdx}.$ If $\delta>0$ is small enough, then this quadratic form is negative definite, i.e., there is a constant $\bar{C}'=\bar{C}'(\delta)>0$ such that 
 \begin{align*} \frac{d}{dt}\mathrm{E}(f(t))\leq -\bar{C}'\left[\mathcal{D}(f)+\int_{\mathbb{T}^d}\int_{B} (f-\Pi f )^2e^{\mathcal{E}} dkdx+\int_{\mathbb{T}^d}\int_{B} (\Pi f-f_{\infty} )^2e^{\mathcal{E}} dkdx\right]. 
 \end{align*}
 It is easy to check that
 $$ \int_{\mathbb{T}^d}\int_{B} (f-\Pi f )^2e^{\mathcal{E}} dkdx+\int_{\mathbb{T}^d}\int_{B} (\Pi f-f_{\infty} )^2e^{\mathcal{E}} dkdx \geq \frac{1}{2}\int_{\mathbb{T}^d}\int_{B} (f-f_{\infty} )^2e^{\mathcal{E}} dkdx.$$
This equation and $\mathcal{D}\geq 0$ imply
\begin{align*} \frac{d}{dt}\mathrm{E}(f(t))\leq -\frac{\bar{C}'}{2}\int_{\mathbb{T}^d}\int_{B} (f-f_{\infty} )^2e^{\mathcal{E}} dkdx. 
 \end{align*}
  The inequality on the right hand side of \eqref{E equiv} implies
 $$\frac{d}{dt}\mathrm{E}(f(t))\leq-\frac{\bar{C}'}{2C_7}\mathrm{E}(f(t)).$$
 Then, Gr\"onwall's inequality implies 
 $$\mathrm{E}(f(t))\leq e^{-\frac{\bar{C}'t}{2C_7}}\mathrm{E}(f_0), \, \, \, \forall\, t\geq 0.$$
 This inequality and \eqref{E equiv} provide $$\int_{\mathbb{T}^d}\int_B\big(f(t,x,k)-f_{\infty}(x,k)\big)^2e^{\mathcal{E}(k)}dkdx\leq c e^{-2\lambda t}\int_{\mathbb{T}^d}\int_B\big(f_0(x,k)-f_{\infty}(x,k)\big)^2e^{\mathcal{E}(k)}dkdx$$ with $c\colonequals \sqrt{\frac{C_7}{C_6}}$ and $\lambda\colonequals \frac{\bar{C}'}{4C_7}.$

 We have obtained this estimate by assuming that $f$ is in $W^{1,\infty}([0,T]\times \mathbb{T}^d\times B)$ for all $T>0.$ This estimate still holds for the distributional solution which is obtained in Theorem \ref{existence}. To prove that one has to approximate $f_0$  with a sequence of smooth functions $f_{0,n}$,  which satisfy \eqref{b1} with some constants  $\mu_{1,n}$ and $\mu_{2,n}$ (cf.\,\cite[Section 3.1]{APT}). Here the constants  $\mu_{1,n}$ and $\mu_{2,n}$ converge to  $\mu_{1}$ and $\mu_2$ as $n\to \infty,$ respectively. Then, by \cite[Theorem 2.1]{Pop.}, the corresponding solution $f_n$ is in $W^{1,\infty}([0,T]\times \mathbb{T}^d\times B)$ for all $T>0.$ Then the above arguments show
 \begin{equation*}
 \int_{\mathbb{T}^d}\int_B\big(f_n(t,x,k)-f_{n,\infty}(x,k)\big)^2e^{\mathcal{E}(k)}dxdk\leq c_n e^{-2\lambda_n t}\int_{\mathbb{T}^d}\int_B\big(f_{0,n}(x,k)-f_{n,\infty}(x,k)\big)^2e^{\mathcal{E}(k)}dxdk,
 \end{equation*}
 where the function $f_{n,\infty}$ is the corresponding global equilibrium and the constants $c_n$ and $\lambda_n$ converge to $c$ and $\lambda$ as $n\to \infty,$ respectively.
The sequence $f_{0,n}$ can be chosen so that the right hand side of the above equation converges to $c e^{-2\lambda t}\int_{\mathbb{T}^d}\int_B\big(f_0(x,k)-f_{\infty}(x,k)\big)^2e^{\mathcal{E}(k)}dxdk$ as $n\to \infty$.  Hence, the left hand side is bounded, which implies that $f_n(t,x,k)$ converges weakly to some function $f(t,x,k)$. One can show the limit function $f$ is the solution obtained in Theorem \ref{existence} (we omit the details). Since the $L^2$-norm is semicontinues with respect to weak convergence, we obtain \eqref{lamda}.   
\end{proof}

 \textbf{Acknowledgement.}
 The author is funded by the Deutsche Forschungsgemeinschaft (DFG, German Research Foundation) under Germany’s Excellence Strategy EXC 2044/2-390685587, Mathematics Münster: Dynamics–Geometry–Structure. \\

\textbf{Data Availability.} There is no data associated with the paper.\\

\textbf{Conflict of interest.} The authors have no conflict of interest to declare.



{}

\end{document}